%% file: Berger_Klosin_Galois_extensions_2018.tex
\documentclass{amsart}
\usepackage{amsmath, amsfonts, amssymb, url}
 \usepackage[all]{xy}

\include{makros}

\include{settings}

\author{Tobias Berger and Krzysztof Klosin}
\title{Modularity of residual Galois extensions and the Eisenstein ideal}
\begin{document}
 
\thanks{The first author's research was supported by the EPSRC Grant EP/R006563/1. The second author was supported by the Young Investigator Grant \#H98230-16-1-0129 from the National
Security Agency, a Collaboration for Mathematicians Grant \#578231  from the Simons Foundation
 and by a PSC-CUNY award jointly funded by the Professional Staff Congress and the City
University of New York.}

\maketitle
\begin{abstract} For a totally real field $F$, a finite extension $\bfF$ of $\bfF_p$  and a Galois character $\chi: G_F \to \bfF^{\times}$ unramified away from a finite set of places $\Sigma \supset \{\fp \mid p\}$ consider the Bloch-Kato Selmer group $H:=H^1_{\Sigma}(F, \chi^{-1})$.  In \cite{BergerKlosin15} it was proved  that the number $d$ of isomorphism classes of (non-semisimple, reducible) residual representations $\ov{\rho}$ giving rise to lines in $H$ which are modular by some $\rho_f$ (also unramified outside $\Sigma$) satisfies $d \geq n:= \dim_{\bfF} H$. This was proved under the assumption that  the order of a congruence module is greater than or equal to that of a divisible Selmer group.
 We show here that if in addition the relevant local Eisenstein ideal $J$ is non-principal, then $d >n$. When $F=\bfQ$ we prove the desired bounds on the congruence module and the Selmer group. We also formulate a congruence condition implying the non-principality of $J$ that can be checked in practice, allowing us to furnish an example where $d>n$. \end{abstract}

\section{Introduction}
Let $p$ be an odd prime and let $\Sigma$ be a finite set of primes of $\bfQ$ containing $p$ where each prime $\ell \in \Sigma$, $\ell \neq p$ satisfies $\ell \not\equiv 1$ (mod $p$). Write $G_{\Sigma}$ for the absolute Galois group of the maximal Galois extension of $\bfQ$ unramified outside of $\Sigma$. Let $E$ be a finite extension of $\bfQ_p$ with integer ring $\Oo$, uniformizer $\varpi$ and $\Oo/\varpi \Oo=\bfF$. Let $\chi: G_{\Sigma} \to \bfF^{\times}$ be a character. Consider a non-split extension of $G_{\Sigma}$-modules
$$ 0 \to \bfF \to \ov{\rho} \to \bfF(\chi) \to 0.$$
In this paper we are interested in the modularity of $\ov{\rho}$ in the following sense: Fix a positive integer $N$ divisible only by the primes in $\Sigma -\{p\}$. We will say that $\ov{\rho}$ is modular (of level $N$)  if there exists a  newform $f$ (of level $N$) giving rise to a (irreducible) Galois representation $\rho_f: G_{\Sigma} \to \GL_2(E)$ and a $G_{\Sigma}$-stable $\Oo$-lattice in the space of $\rho_f$ such that with respect to this lattice the mod $\varpi$ reduction $\ov{\rho}_f$ of $\rho_f$ is isomorphic to $\ov{\rho}$ (as representations). 

This is a very strong notion of modularity for two reasons: \begin{enumerate} \item we require that $\ov{\rho}_f \cong \ov{\rho}$ rather than simply $\tr \ov{\rho}_f = \tr \ov{\rho}$ and
\item we do not allow $\rho_f$ to be ramified at primes outside of $\Sigma$. 
\end{enumerate}
The requirement (2) stands in contrast with the work of Hamblen and Ramakrishna \cite{HamblenRamakrishna08} who prove modularity of such $\ov{\rho}$ by $\rho_f$ in the sense of (1), but allow for additional ramification of $\rho_f$.
More specifically, they show the existence of  a characteristic zero lift $\rho: G_{\Sigma'} \to \GL_2(\Oo)$ of $\ov{\rho}$ for some set $\Sigma' \supset \Sigma$ and then use the modularity theorem of Skinner and Wiles \cite{SkinnerWiles99} to conclude modularity of $\ov{\rho}$. 

To the best of our knowledge the question of modularity of $\ov{\rho}$ in our strong sense has never been studied despite being rather natural. (In the semi-simple reducible case such an analysis was carried out by Billerey and Menares in \cite{BillereyMenares18} using a different method.) While we are not able to prove that all  $\ov{\rho}$ as above are modular in this sense, this is perhaps not to be expected. In particular not all such extensions will in general be modular if we fix the level $N$ as there are only finitely many forms of fixed level (we also fix the weight by imposing a condition on the determinant). So, in particular enlarging $\bfF$ (which increases the number of isomorphism classes of $\ov{\rho}$) will produce non-modular extensions. This prompts an intriguing question: given $N$ how many of the extensions $\ov{\rho}$ are modular of level $N$? In this article we give a lower bound on this number when $\ov{\rho}$ is in the image of the Fontaine-Laffaille functor as we now explain. 
While we limit most of our discussion here for simplicity to the case of $\bfQ$, we prove some of our results for a general totally real field $F$ (see below).

 Any isomorphism class $\ov{\rho}$ in the category of representations gives rise to a line in the residual Bloch-Kato Selmer group  $H^1_{\Sigma}(\bfQ, \chi^{-1})$ (where we do not impose any conditions on primes in $\Sigma$ other than $p$). We showed in \cite{BergerKlosin15} that under some assumptions the group $H^1_{\Sigma}(\bfQ, \chi^{-1})$ has a basis consisting of modular extensions, i.e., that at least $n:=\dim H^1_{\Sigma}(\bfQ, \chi^{-1})$ such isomorphism classes of $\ov{\rho}$ are modular. 

Improving this bound (which is the main goal of this paper) is a tougher problem and we show it is related to the structure of the Eisenstein ideal $J$ of the (local) cuspidal Hecke algebra $\bfT$. We obtain the most satisfactory answer for $F=\bfQ$. In this case we show that if $J$ is not principal and  the Selmer group $H^1_{\Sigma}(\bfQ, \chi)$ (``for  extensions in  the opposite order'' of characters to the one in $\ov{\rho}$)  is one-dimensional, then  the number of modular isomorphism classes of the representations $\ov{\rho}$ is strictly larger than $n$ (under some restrictions on $\Sigma$ and $\chi$) - cf.~Corollary \ref{sum up 2}.

One of the immediate consequences of our results is that if $J$ is not principal then $\dim H^1_{\Sigma}(\bfQ, \chi^{-1}) >1$ (since in a one-dimensional Selmer group there is only one line!). We note that  Wake and Wang-Erickson \cite{WakeWangErickson18preprint} give a cohomological lower bound on the number of generators of the Eisenstein ideal for modular forms of weight 2 and trivial nebentypus. A side effect of our result (but one that applies to the case of $k>2$ or $k=2$ and non-trivial nebentypus, so not the case studied in \cite{WakeWangErickson18preprint})  is that it provides a condition in the converse direction, i.e., $J$ not principal implies $\dim H^1_{\Sigma}>1$.

In the process of proving Corollary \ref{sum up 2} (i.e., when $F=\bfQ$) we establish a lower  bound on the congruence module $\bfT/J$ by a certain Bernoulli number with correction factors. Previous results of this kind include Theorem 5.1 in \cite{SkinnerWiles97},  which applies in the case of $k=2$ and non-trivial nebentypus and an analogous result of Mazur \cite{Mazur78}, Proposition II.9.7 (for $k=2$, prime level and trivial nebentypus). 
We also establish a corresponding upper bound on the relevant Bloch-Kato Selmer group which together with the $\bfT/J$-bound are key for the existence of a modular basis of $H^1_{\Sigma}(\bfQ, \chi^{-1})$.  We also prove bounds on other Selmer groups that allow one to check when $\dim_{\bfF}H^1_{\Sigma}(\bfQ, \chi)=1$ and $\dim_{\bfF}(\bfQ, \chi^{-1})>1$ (the case when our theorem is interesting). 

For a general $F$ we obtain a similar result. However the existence of corresponding bounds on $\bfT/J$ and the Selmer group, while expected to hold, is not yet known.

Let us discuss the organization of the paper. In section \ref{Setup} we establish basic notation and facts regarding Selmer groups and Fontaine-Laffaille representations. In section \ref{The rings Ttau} we study the relevant Hecke algebra $\bfT$ along with its quotients $\bfT_{\tau}$ corresponding to newforms whose Galois representations reduce to different isomorphism classes of (reducible) residual representations $\tau$. We also define the Eisenstein ideal $J$ and prove a preliminary result guaranteeing the existence of more than $n$ modular Galois extensions (Proposition \ref{main1}). In section \ref{Ideal of reducibility and its principality} we introduce and study the ideals of reducibility of the Galois representations $\rho_{\tau}: G_{\Sigma} \to \GL_2(\bfT_{\tau})$ (whose existence we prove) showing their principality under the assumption that $\dim_{\bfF}H^1_{\Sigma}(\bfQ, \chi)=1$. This allows us to strengthen Proposition \ref{main1} to Theorem \ref{main3}. In section \ref{FQ} we strengthen Theorem \ref{main3} further in the case $F=\bfQ$ by proving an equality between the orders of $\bfT/J$ and the relevant divisible Selmer group. In section \ref{Analysis} we establish bounds on certain Selmer groups allowing us (among other things) to verify the condition $\dim_{\bfF}H^1_{\Sigma}(\bfQ, \chi)=1$ for an example which we discuss in section \ref{Ex}.

We would like to thank David Spencer for informing us about \cite{BillereyMenares18} and \cite{Spencer18}. We are also grateful to Neil Dummigan and Carl Wang-Erickson for helpful comments.

\section{Setup}\label{Setup}
 Let $F$ be a  totally real field  
and $p>2$ a prime with $p \nmid \# \Cl_F$  and $p$ unramified in $F/\bfQ$. Let $\Sigma$ be a finite set of finite places of $F$ containing all the places
lying over $p$. Assume that if $\fq \in \Sigma$, then $N \fq \not\equiv 1$ (mod $p$). Let $G_{\Sigma}$ denote the Galois
group $\Gal(F_{\Sigma}/F)$, where $F_{\Sigma}$ is the maximal extension of $F$ unramified outside $\Sigma$.  For every prime $\fq$ of $F$ we fix compatible embeddings $\ov{F} \hookrightarrow \ov{F}_{\fq} \hookrightarrow \bfC$ and write $D_{\fq}$ and $I_{\fq}$ for the corresponding decomposition and inertia subgroups of $G_F$ (and also their images in $G_{\Sigma}$ by a slight abuse of notation). Let $E$ be a (sufficiently large) finite extension of $\bfQ_p$ with ring of integers $\Oo$ and residue field $\bfF$. We fix a choice
of a  uniformizer $\varpi$. We will write $\epsilon$ for the $p$-adic cyclotomic character, $\ov{\epsilon}$ for its mod $p$ reduction, and $\omega$ for the Teichm\"uller lift of $\epsilon$. For a local ring $A$ we write $\fm_A$ for its maximal ideal.

\subsection{Fontaine-Laffaille representations} \label{s2.1}
Let $n$ be any positive integer. Suppose $$r: G_{\Sigma} \to {\rm GL}_n(\bfF)$$ is a continuous homomorphism.

We recall from \cite{ClozelHarrisTaylor08} p. 35 the definition of a \emph{Fontaine-Laffaille} representation: Let $\fp \mid p$
and $A$ be a local complete Noetherian $\bfZ_p$-algebra with residue field $\bfF$. A representation $\rho: D_{\fp} \to {\rm GL}_n(A)$ is Fontaine-Laffaille if for each Artinian quotient
$A'$ of $A$, $\rho \otimes A'$ lies in the essential image of the Fontaine-Laffaille functor $\mathbf{G}$ (for its definition see e.g. \cite{BergerKlosin13} Section 5.2.1). We also call a continuous finite-dimensional $G_{\Sigma}$-representation $V$ over $\bfQ_p$  Fontaine-Laffaille if, for all primes $\fp \mid p$,  it is crystalline and
${\rm Fil}^0 D=D$ and ${\rm Fil}^{p-1} D=(0)$ for the filtered vector space $D=(B_{\rm crys} \otimes_{\bfQ_p} V)^{D_{\fp}}$ defined by Fontaine (for details see again \cite{BergerKlosin13} Section 5.2.1).

For $j\in \{1,2\}$ let $\tau_j: G_{\Sigma} \to \GL_{n_j}(\bfF)$ be an absolutely irreducible continuous representation. Assume that $\tau_1 \not\cong \tau_2$. Consider the set of isomorphism classes of $n$-dimensional residual Fontaine-Laffaille  representations of the form: \be \label{form1}\tau=\bmat \tau_1 & * \\ & \tau_2\emat: G_{\Sigma} \rightarrow \GL_n(\bfF),\ee which are non-semi-simple ($n=n_1+n_2$).

\subsection{Selmer groups} For a $p$-adic $G_{\Sigma}$-module $M$ (finitely generated or cofinitely generated over $\Oo$ - for precise definitions cf.~\cite{BergerKlosin13}, section 5) we define the Selmer group $H^1_{\Sigma}(F, M)$ to be the subgroup of $H^1_{\rm cont}(F_{\Sigma}, M)$ consisting of cohomology classes which are crystalline in the sense of Bloch-Kato  at all primes $\fp$ of $F$ dividing $p$, i.e. 
$$H^1_{\Sigma}(F, M)=\ker(H^1(G_{\Sigma}, M) \to  \prod_{\fp \mid p}(H^1(F_\fp, M)/H^1_f(F_\fp, M)).$$

For $G_{\Sigma}$-modules $M$ occurring as $\Oo$-lattices $T$ in $E$-vector spaces $V$ or as divisible modules $V/T$ the crystalline conditions $H^1_f(F_\fp, M)$ are as defined by Bloch-Kato in \cite{BlochKato90} (cf.~also section 1 in \cite{Rubin00}). For  $G_{K}$-modules $M$ of finite cardinality we use Fontaine-Laffaille theory to define the local condition: 
If $K$ denotes an unramified extension of $\bfQ_p$ then  if $M$ is in the essential image of the Fontaine-Laffaille functor $\mathbf{G}$ we define $H^1_f(K,M)$ as the image of ${\rm Ext}^1_{\mathcal{M
F}_\Oo}(1_{\rm FD}, D)$ in $H^1(K,M)\cong {\rm Ext}^1_{\Oo[G_K]}(1,M)$, where $\mathcal{M
F}_\Oo$ is the category of filtered Dieudonn\'e
modules, $\mathbf{G}(D)=M$ and $1_{\rm FD}$ is the unit filtered Dieudonn\'e
module defined in Lemma 4.4 of \cite{BlochKato90}. 
Note that we place no restrictions at the primes in $\Sigma$ that do not lie over $p$. For more details cf.~[loc.cit.].

\section{The rings $\bfT_{\tau}$} \label{The rings Ttau}

\begin{prop} \label{genRibet} Suppose $\rho: G_{\Sigma} \rightarrow \GL_n(E)$ is irreducible and satisfies  \be \label{semis} \ov{\rho}^{\rm ss} \cong \tau_1 \oplus \tau_2,\ee where $\ov{\rho}^{\rm ss}$ denotes the semi-simplification of any residual representation of $\rho$. Then there exists a lattice  inside $E^{n}$ so that with respect to that lattice
the mod $\varpi$ reduction $\ov{\rho}$ of $\rho$ has the form $$\ov{\rho}=\bmat \tau_1 & * \\ 0 & \tau_2\emat$$ and is
non-semi-simple. \end{prop}

\begin{proof} This argument goes back to Ribet and in this form is a special case of  \cite{Urban01}, Theorem 1.1, where the ring $\mB$ in [loc.cit.] is the discrete valuation ring $\Oo$.
\end{proof}

 For $\tau$ as in \eqref{form1} let $\Phi_{\tau, E}$ be the set of isomorphism classes of Fontaine-Laffaille at $\fp \mid p$ Galois representations $\rho: G_{\Sigma} \to \GL_n(E)$ such that there exists a $G_{\Sigma}$-stable lattice $L$ in the space of $\rho$ so that the mod $\varpi$-reduction of $\rho_L$ equals $\tau$. 
The following is a higher-dimensional analogue of Lemma 2.13(ii) from \cite{SkinnerWiles99}:
\begin{prop}[\cite{BergerKlosin15}, Proposition 3.2] \label{SW2.13} One has $\Phi_{\tau,E} \cap \Phi_{\tau',E}=\emptyset$ if $\tau \not\cong \tau'$. \end{prop}
For the rest of this section set $n=2$, $\tau_1=1$ and $\tau_2=\chi = \psi \ov{\epsilon}^{k-1}$, where $\psi$ is unramified at $p$ and $k$ is an integer such that $2 \leq k \leq p-1$. Write $\tilde{\psi}$ for the Teichm\"uller lift of $\psi$ and set $\tilde{\chi} = \tilde{\psi} \epsilon^{k-1}$.  

Let $\fN$ be an ideal of $\OF$ divisible only by primes in $\Sigma$ which do not lie over $p$. 
We consider the space $\mS_k(\fN, \tilde{\psi})$ of cuspidal Hilbert modular forms (over the field $F$)  of parallel weight $k \geq 2$, level $\Gamma_0(\fN)$ and  character $\tilde{\psi}$. 
Let $\bfT'$ be the $\Oo$-subalgebra 
of $\End_{\bfC} \mS_k(\fN, \tilde{\psi})$ generated by the Hecke operators $T_{\fq}$ for all $\fq \not\in\Sigma$. 
Set $J'$ to be the ideal of $\bfT'$ generated by the set $\{T_{\fq}-(1+\tilde{\psi}(\fq)(N\fq)^{k-1}) \mid \fq \not\in \Sigma\}$. 
Let $\fm$ be a maximal ideal of $\bfT'$ containing $J'$ and set $\bfT$ to be the completion of $\bfT'$ at the ideal $\fm$.
\begin{definition}  We will call $J:=J'\bfT$ the (local) \emph{Eisenstein ideal} (associated to $\tilde{\psi}$). \end{definition}

We refer to the surjective $\Oo$-algebra homomorphisms $\lambda: \bfT \twoheadrightarrow \Oo$ as \emph{Hecke eigensystems}. For each such $\lambda$ we denote by $\tilde{\tau}_{\lambda} : G_{\Sigma} \to \GL_2(E)$ the corresponding (irreducible) Galois representation.  Using Proposition \ref{genRibet} we see that there exists a lattice in $E^2$ with respect to which $\tilde{\tau}_{\lambda}$ is valued in $\GL_2(\Oo)$ such that its mod $\varpi$ reduction $\ov{\tilde{\tau}}_{\lambda}$ is non-semisimple. Proposition \ref{SW2.13} guarantees that the isomorphism class of $\ov{\tilde{\tau}}_{\lambda}$ is independent of the choice of such a lattice. In view of this we will simply write $\tau_{\lambda}$ for the \emph{non-semi-simple residual} Galois representation attached to $\lambda$ (well-defined up to isomorphism). We write $\bfT_{\tau}$ for the image of the canonical map $$\bfT \to \prod_{\lambda: \tau_{\lambda} \cong \tau} \Oo,$$ i.e., the quotient of $\bfT$ corresponding to all Hecke eigensystems whose associated residual non-semisimple Galois representations are isomorphic to $\tau$. If no $\tau_{\lambda}$ is isomorphic to $\tau$ we set $\bfT_{\tau}=0$. We will denote by $J_{\tau}$ the image of $J$ in $\bfT_{\tau}$. 
 
\begin{rem} \label{finiteness1} It is clear that $\bfT$ and $\bfT_{\tau}$ are finitely generated $\Oo$-modules. Furthermore, $\#\bfT/J < \infty$ as otherwise, as we show below, there would exist a surjective $\Oo$-algebra map $\bfT \to \Oo$ factoring through $\bfT/J$. The existence of such a map would violate the Ramanujan bounds. For the sake of contradiction suppose $\#\bfT/J = \infty$. Then $\bfT/J = \Oo^s \times T$ as an $\Oo$-module with $T$  finite and $s>0$. 
Hence $\bfT/J$ is not of finite length as an $\Oo$-module, and it is easy to see that it is also not of finite length as a module over itself. Since $\bfT$ is Noetherian, it follows that there is a prime ideal $\fp$ of $\bfT/J$ which is not maximal (cf.~Theorem 2.14 in \cite{Eisenbud}), hence $\bfT/(J+\fp)$ is an infinite domain (as all finite domains are fields).  This implies that the structure map $\Oo \to \bfT/(J+\fp)$ is injective (as $\bfT$ is a finitely generated $\Oo$-module), and so the domain $\bfT/(J+\fp)$ is finite over $\Oo$, thus  we may assume it equals $\Oo$ as $\Oo$ is assumed to be sufficiently large. Hence  the canonical map $\bfT/J \twoheadrightarrow \bfT/(J + \fp)=\Oo$ gives us the $\Oo$-algebra surjection.\end{rem}

Note that isomorphism classes of Fontaine-Laffaille residual representations $\tau: G_{\Sigma} \to \GL_2(\bfF)$ such that $\tau = \bmat 1 &* \\ & \chi\emat$  are in one-to-one correspondence with lines in $H^1_{\Sigma}(F, \chi^{-1})$. Since $2 \leq k<p$ the representations $\tilde{\tau}_{\lambda}$ (and $\tau_{\lambda}$) are Fontaine-Laffaille at primes lying over $p$. 
\begin{definition} \label{defmod} We will say that  (an isomorphism class of) $\tau=\bmat 1&*\\ &\chi\emat: G_{\Sigma} \to \GL_2(\bfF)$ is \emph{modular} if there exists $\lambda: \bfT \to \Oo$ such that $\tau_{\lambda} \cong \tau$ (in other words, if $\bfT_{\tau} \neq 0$). \end{definition}
\begin{rem} Note that the requirement in Definition \ref{defmod} is stronger than the usual definition of modularity which simply asks that $\tr \tau = \tr \tilde{\tau}_{\lambda}$ for $\tilde{\tau}_{\lambda}: G_{\Sigma} \to \GL_2(E)$. \end{rem}

\begin{thm} [Corollary 4.8 in \cite{BergerKlosin15}] \label{BK main}   Suppose that $\# H^1_{\Sigma}(F, \tilde{\chi}^{-1}\otimes E/\Oo) \leq \#\bfT/J$. 
Then there exists a basis $\mB$ of $H^1_{\Sigma}(F, \chi^{-1})$ such that each $\tau \in \mB$ is modular. \end{thm}
\begin{proof} Let us only explain why Assumption 2.4 in \cite{BergerKlosin15} used in Corollary 4.8 therein is satisfied. For this it is enough to show that there are no non-trivial infinitesimal deformations of $1$, respectively $\chi$. This can be proved exactly as \cite{BergerKlosin13} Proposition 9.5 since $p \nmid \#\Cl_F$. \end{proof}

\begin{rem} The assumption that $\# H^1_{\Sigma}(F, \tilde{\chi}^{-1}\otimes E/\Oo) \leq  \#\bfT/J$ is used in the proof of Corollary 4.8 in \cite{BergerKlosin15}. The left-hand side of the inequality encodes certain  crystalline $G_{\Sigma}$-extensions of torsion $\Oo$-modules  while the right-hand side encodes corresponding modular extensions (arising from Eisenstein congruences). Hence it can be viewed as in some sense ensuring an abundance of reducible modular deformations of appropriate type. Roughly speaking, the Selmer group on the left hand side should be bounded by a certain $L$-value  by virtue of the relevant case of the Bloch-Kato Conjecture. Then the inequality in the assumption reflects the belief that Eisenstein congruences should be controlled by the same $L$-value. In section \ref{FQ} we will prove that these inequalities are often satisfied when $F=\bfQ$. 
\end{rem}

Let $\fT$ denote the set of isomorphism classes of residual Galois representations of the form \eqref{form1}. Let $\fT_{\rm mod}$ be the subset of $\fT$ consisting of isomorphism classes which are modular. Note that by Proposition \ref{genRibet} each element of $\fT_{\rm mod}$ can be identified with a line in $H^1_{\Sigma}(\bfQ, \chi^{-1})$ and Theorem \ref{BK main} gives a sufficient condition for the existence of at least $\dim_{\bfF}H^1_{\Sigma}(F, \chi^{-1})$-many such lines. These lines span the Selmer group, but a natural question to ask is if one could strengthen the conditions of Theorem \ref{BK main} to guarantee the existence of even more modular lines. This is achieved by the following proposition which is the first main result of this paper.

\begin{prop} \label{main1}   Suppose that  $\# H^1_{\Sigma}(F, \tilde{\chi}^{-1}\otimes E/\Oo) \leq  \#\bfT/J$. If $J_{\tau}$ is principal for every $\tau \in \fT_{\rm mod}$ but $J$ is not principal, then the set $\fT_{\rm mod}$ of modular isomorphism classes has cardinality strictly greater than $\dim_{\bfF} H^1_{\Sigma}(F, \chi^{-1})$. \end{prop}

\begin{proof} Let us first note that by Remark \ref{finiteness1} we have that $\bfT$ is finitely generated as an $\Oo$-module and $\#\bfT/J < \infty$, hence the results of \cite{BergerKlosin15} and \cite{BergerKlosinKramer14} apply.
By Proposition 5.1 in \cite{BergerKlosin15} we have that $$\#\bfT/J \geq \#\prod_{\tau \in \fT_{\rm mod}} \bfT_{\tau}/J_{\tau}.$$ By Theorem \ref{BK main} we know that there exists a modular basis $\mB$ of $H^1_{\Sigma}(F, \chi^{-1})$, so in particular $\#\fT_{\rm mod} \geq \dim_{\bfF}H^1_{\Sigma}(F, \chi^{-1})$. Suppose that in fact equality holds. Since any modular extension gives rise to an element of $\fT_{\rm mod}$, we see that any other modular basis of $H^1_{\Sigma}(F, \chi^{-1})$ must be obtained from  $\mB$ by scaling its elements, i.e., $\mB$ is `projectively unique' in the terminology of \cite{BergerKlosin15}. Then by Proposition 5.4 in \cite{BergerKlosin15} we get that $\#\bfT/J =\#\prod_{\tau \in \fT_{\rm mod}} \bfT_{\tau}/J_{\tau}.$ This however implies that $J$ is principal by Corollary 2.7 of \cite{BergerKlosinKramer14} - note that principality of $J_{\tau}$ is necessary for the application of the corollary (cf.~p. 73 of \cite{BergerKlosinKramer14}). 
\end{proof}

For future use we note that the opposite inequality $\# H^1_{\Sigma}(F, \tilde{\chi}^{-1}\otimes E/\Oo) \geq  \#\bfT/J$ always holds:
\begin{prop} \label{3rdinequality} One has
$$\# H^1_{\Sigma}(F, \tilde{\chi}^{-1}\otimes E/\Oo) \geq  \#\bfT/J.$$
\end{prop}

\begin{proof}
This is proved by applying Urban's lattice construction, as explained in the proof of \cite{BergerKlosin15} Lemma 4.4 (we do not need the assumptions 2.5 and 4.2 there as we just want an inequality of orders).
\end{proof}

In the next section we show that if one assumes one-dimensionality of the ``opposite'' Selmer group $H^1_{\Sigma}(F, \chi)$ then principality of each $J_{\tau}$ follows. 

\section{Ideal of reducibility and its principality} \label{Ideal of reducibility and its principality}

Let $G$ be a group and $A$ be a complete Noetherian local $\Oo$-algebra (with residue field $\bfF$) which is reduced.  Set $R=A[G]$. Let $\tau_1, \tau_2: G \to \GL_{n_i}(\bfF)$ be two absolutely irreducible representations with $\tau_1 \not\cong \tau_2$. 
Set $n:=n_1+n_2$ and assume that $n!$ is invertible in $A$.
Let $T$ be a (residually multiplicity free) pseudo-representation $T: R \to A$ of dimension $n$. Following \cite{BellaicheChenevierbook} we define the \emph{ideal of reducibility of $T$} to be the  smallest ideal $I$ of $A$ such that $T = T_1 + T_2$ mod $I$, where $T_1$, $T_2$ are pseudo-representations with the property that $T_i = \tr \tau_i$ mod $\fm_A$. 
Let $\rho: R \to M_{n}(A)$ be an $A$-algebra homomorphism.  Suppose that the mod $\fm_A$ reduction $\ov{\rho}: R \to M_{n}(\bfF)$ of $\rho$ has the form $$\ov{\rho} =\bmat \tau_1 & * \\ & \tau_2 \emat $$ and is non-semi-simple. We define the ideal of reducibility of $\rho$ to be the ideal of reducibility of the pseudo-representation $\tr \rho$.

 Write $\mF:=\textup{Frac}(A)$, the total ring of fractions of $A$, which is a finite product of fields $\prod_{i=1}^s A_i$ (cf.~e.g., \cite{BellaicheChenevierbook}, section 1.7). 
Fix $S_{ij} \subset \Ext^1_{\bfF[G]}(\tau_i, \tau_j)$  one-dimensional subspaces for $(i,j) \in \{(1,2), (2,1)\}$. 
Assume that the pseudo-representation $\tr \rho_i: R \to A_i$ is absolutely irreducible for every $i=1,2, \dots, s$. Moreover, assume that $\ov{\rho}: R \to M_{n}(\bfF)$ which factors through $\bfF[G] \to M_{n}(\bfF)$ gives rise to a non-trivial element in $S_{21}$.

\begin{prop} [\cite{BellaicheChenevierbook}, Proposition 1.7.4] \label{1.7.4} One has $$\dim_{\bfF} \Ext^1_{(R/\ker \rho)/\fm_A(R/\ker \rho)} (\tau_2, \tau_1)=1.$$ \end{prop}

\begin{proof} Let us only note that Proposition 1.7.4 in \cite{BellaicheChenevierbook} concerns $\ker T$ instead of $\ker \rho$. However, it follows from Proposition 1.6.4 of \cite{BellaicheChenevierbook} along with our assumption on absolute irreducibility of $\tr \rho_i$ that $\ker \rho = \ker T$. \end{proof}

The goal of this section is to give a sufficient condition guaranteeing that $I$ is principal. Before we begin  let us briefly explain the method.   If the dimension of $\Ext^1_{(R/\ker \rho)/\fm_A(R/\ker \rho)} (\tau_1, \tau_2)$ (``opposite direction'') is also one, $I$ would be principal by Proposition 1.7.5 of \cite{BellaicheChenevierbook}. To prove this  
we use Urban's  construction to obtain an $A$-module $\mT \oplus A$ together with a $G$-action which modulo $\fm_A$ gives a non-split extension in the ``opposite direction''. If $\mT=A$, then this  extension is a reduction of a representation of $G$ into $\GL_2(A)$ and Proposition 1.7.4 in \cite{BellaicheChenevierbook} gives us the desired one-dimensionality. In the proof of Theorem \ref{prin} we formulate a condition that allows us to conclude that $\mT/\fm_A\mT=\bfF$ and essentially deduce from this  that $\mT=A$ by Nakayama's Lemma.

From now on assume that $A$ is finite over $\Oo$. We will later apply this for $A=\bfT_{\tau}$ for which this assumption is satisfied  (cf.~Remark \ref{finiteness1}). 
Then by Theorem 1.1 in \cite{Urban01} there exists an $A$-lattice $\mL$ in $\mF^{n}$ and an $A$-lattice $\mT$ in $\mF$ such that \be \label{lattice} 0 \to \tau_2 \otimes_A \mT/\fm_A\mT \to \mL\otimes_A \bfF\to \tau_1 \otimes_A \bfF \to 0.\ee
As in \cite{Urban01} (see also \cite{Klosin09}, p. 159-160) we get a cocycle 
$c \in H^1(G, \Hom(\tau_1, \tau_2)\otimes \mT/\fm_A\mT)$ and a 
 map $$\iota: \Hom(\mT/\fm_A\mT, \bfF) \to \Ext^1_{\bfF[G]}(\tau_1, \tau_2) = H^1(G, \Hom(\tau_1, \tau_2)), \quad f \mapsto (1\otimes f)(c),$$ 
 which is injective by Lemma 4.5 in \cite{BergerKlosin15}.
\begin{thm} \label{prin} If the image of $\iota$ lies in $S_{12}$, then $I$ is principal. \end{thm}
\begin{proof} We have $\mT/\fm_A \mT = \bfF^s$ for some $s \in \bfZ_+$. Since $S_{12}=\bfF$, the injectivity of $\iota$ implies that $s=1$. Hence (\ref{lattice}) itself is an element of $S_{12}$. Moreover by a  complete version of Nakayama's Lemma, $\mT$ is generated by 1 element, say $x \in \mT$, as an $A$-module. We claim that this implies that $\mT=A$. Indeed, consider the $A$-module map $\phi: A \twoheadrightarrow \mT$ given by $r \mapsto rx$. We will show that this map is injective. Suppose $a$ is in the kernel. Then $a$ annihilates $\mathcal{T}$. However, by definition of $\mT$ and the fact that $A$
is reduced and hence embeds into its ring of fractions $\mF$ we can consider $x$ and $a$ as elements of $\mF =\prod_i A_i$, i.e., write them as $a = (a_1, a_2, \dots,  a_s)$ and
$x = (x_1, x_2, \dots, x_s)$. We want to show that $a =0$.

Let $\mJ$ be the set of $i$ such that $a_i \neq  0$.
First note that if $j \in \mJ$, then $xA \otimes_A A_j = 0$. Indeed, if $j \in \mJ$, then since $ax=0$, we must have $x_j=0$, so $x\alpha \otimes 1 = x\alpha a \otimes 1/a_j=0$ for all $\alpha \in A$. 
Secondly note that if $j \not\in \mJ$, then $xA \otimes_A A_j$ is of dimension $\leq 1$ as an $A_j$-vector space. Indeed, let $\sum_k x \alpha_k \otimes  \beta_k \in xA \otimes_A A_j$ and write $\pi_j$ for the map $A \to A_j$. Then $$\sum_k x \alpha_k \otimes  \beta_k  = \sum_k x \otimes \pi_j(\alpha_k) \beta_k = x \otimes \left( \sum_k \pi_j(\alpha_k) \beta_k\right) = (x\otimes 1) \cdot \left(\sum_k \pi_j(\alpha_k )\beta_k\right) ,$$ hence indeed $xA \otimes_A A_j$ is spanned over $A_j$ by $x \otimes 1$.

Thus we get  $$\mT 
\otimes_A \mF = xA \otimes_A \prod_i A_i = \prod_i xA \otimes_A A_i = \prod_{i \not\in \mJ} xA \otimes_A A_i $$ and each piece of the product is either 0 or $A_j$. Since $\mT$ is a lattice we must have $\mT\otimes_A\mF=\mF=\prod_i A_i$, and this forces $\mJ=\emptyset.$

Hence $\mL \cong A^{n}$, so (\ref{lattice}) is the reduction of a representation $R \to M_{n}(A)$. Thus by \cite{BellaicheChenevierbook}, Proposition 1.7.4, we get that $$\dim_{\bfF} \Ext^1_{(R/\ker \rho)/\fm_A(R/\ker \rho)} (\tau_1, \tau_2)=1$$ and thus by [loc.cit.], Proposition 1.7.5 the ideal $I$ is principal. 
\end{proof}

\begin{lemma} \label{imageT} Let $\tau \in \fT_{\rm mod}$. There exists a representation $\rho_{\tau}: G_{\Sigma} \to \GL_2(\bfT_{\tau})$ that reduces to $\tau$ modulo $\fm_{\bfT_{\tau}}$. 
\end{lemma} 
\begin{proof}  Consider the representation $$\rho'_{\tau}: G_{\Sigma} \to \prod_{\lambda : \tau_{\lambda} \cong \tau} \GL_2(\Oo)\subset \GL_2({\rm Frac} (\bfT_{\tau}))$$ given by the representations $\tilde{\tau}_{\lambda}$. We now proceed as in the proof of Theorem 6.2 in \cite{BergerKlosin15} replacing $R_{\tau}^{\rm tr, 0}$ there with $\bfT_{\tau}$. We only give a brief outline here as the argument is essentially identical. Using Theorem 4.1 in \cite{BergerKlosin15} we deduce the existence of a Galois invariant lattice $\mL$ in the representation space ${\rm Frac}(\bfT_{\tau})^2$  of $\rho'_{\tau}$ and a $\bfT_{\tau}$-lattice $\mT_{\tau} \subset {\rm Frac}(\bfT_{\tau})$ which fits into the exact sequence \be \label{6.2} 0 \to \mT_{\tau}/I_{\tau}\mT_{\tau} \to \mL \otimes_{\bfT_{\tau}}\bfT_{\tau}/I_{\tau} \to \tilde{\chi}\otimes_{\Oo} \bfT_{\tau}/I_{\tau} \to 0,\ee where $I_{\tau}$ is the ideal of reducibility of the pseudo-representation $\tr \tau$. 

As in the proof of Theorem 6.2 in \cite{BergerKlosin15} one notes that $\mL \cong \mT_{\tau} \oplus \bfT_{\tau}$ as $\bfT_{\tau}$-modules and then shows that $\mT_{\tau} /I_{\tau}\mT_{\tau} \otimes_{\bfT_{\tau}}\bfF \cong \bfF$, so we get $\mT_{\tau} = \bfT_{\tau}$ as in the proof of Theorem \ref{prin} above. Thus  \eqref{6.2} gives rise to a representation $\rho_{\tau}$ as in the statement of the Lemma. 
\end{proof}
\begin{rem} We note that Lemma \ref{imageT} does not imply that there is a representation of $G_{\Sigma}$ into $\GL_2(\bfT)$. In the residually irreducible case this is in fact the case (cf.~Lemma 3.27 in \cite{DDT}). Also if one assumes that $\tau$ is unique (i.e., that there is only one isomorphism class of non-semisimple residual representations with semi-simplification $1 \oplus \chi$)
 this is also true and follows from the fact that in this case the universal deformation ring is generated by traces (cf.~Corollary 3.2 in \cite{SkinnerWiles97} and Proposition 7.13 in \cite{BergerKlosin13}). However, in general (when several different $\tau$s exist), this need no longer be the case. Lemma \ref{imageT} can be viewed as providing a substitute for the existence of  a representation into $\GL_2(\bfT)$  when one fixes a particular residual representation $\tau$. However, while $\bfT_{\tau}$ is a quotient of $\bfT$, in general there is no natural map $\bfT_{\tau} \to \bfT$.
\end{rem}

Using Lemma \ref{imageT} we can write  $I_{\tau}$ for the ideal of reducibility of  $\rho_{\tau}$. Let us now apply Theorem \ref{prin} to our situation with $A=\bfT_{\tau}$. Note that the cuspidality of $\bfT_{\tau}$ ensures that the assumption of absolute irreducibility of the generic components of $\rho_{\tau}$ is satisfied. 
\begin{lemma} \label{Eis=red} One has $J_{\tau} = I_{\tau}$. \end{lemma}
\begin{proof} This can be proved like Lemma 2.9 in \cite{BergerKlosin15}. \end{proof}

\begin{prop} \label{our sit} If  $\dim_{\bfF} H^1_{\Sigma}(F, \chi)=1$ then $I_{\tau}$ is a principal ideal. \end{prop}

\begin{proof} Because $\tau$ is an actual representation, Proposition \ref{1.7.4}  gives us that $$\dim_{\bfF} \Ext^1_{(\bfT_{\tau}[G_{\Sigma}]/\ker \rho_\tau)/\fm_{\tau}(\bfT_{\tau}[G_{\Sigma}]/\ker \rho_\tau)} (\chi, 1)=1.$$ 
We set $A=\bfT_{\tau}$, $G=G_{\Sigma}$ and $S_{21}=H^1_{\Sigma}(\bfQ, \chi)$. 
The claim follows from Theorem \ref{prin} and Lemma \ref{image1} below. \end{proof}
\begin{lemma}\label{image1} The image of $\iota: \Hom(\mT/\fm_{\bfT}\mT, \bfF) \to H^1(G_{\Sigma}, \chi)$ is contained in $H^1_{\Sigma}(F, \chi)$. \end{lemma}

\begin{proof} This follows from Lemma 4.5 in \cite{BergerKlosin15}, except that we do not need the assumptions 2.5 and 4.2 there, as we do not claim surjectivity of $\iota$ here.  
\end{proof}
Combined with Proposition \ref{main1} we obtain the following result. 

\begin{thm} \label{main3} Suppose that $\# H^1_{\Sigma}(F, \tilde{\chi}^{-1}\otimes E/\Oo)\leq \#\bfT/J$. If  $\dim_{\bfF} H^1_{\Sigma}(F, \chi)=1$ and the Eisenstein ideal $J$ is not principal, then $\#\fT_{\rm mod}> \dim_{\bfF}H^1_{\Sigma}(F, \chi^{-1})$. \end{thm}

We end this section by stating a cohomological criterion guaranteeing the principality of the Eisenstein ideal.
 
\begin{cor} \label{main4} Suppose that $\# H^1_{\Sigma}(F, \tilde{\chi}^{-1}\otimes E/\Oo)\leq \#\bfT/J$. Suppose furthermore that $\dim_{\bfF} H^1_{\Sigma}(F, \chi)=\dim_{\bfF} H^1_{\Sigma}(F, \chi^{-1})=1$. Then $J$ is principal. \end{cor}
\begin{proof} In this case there is only one line in $H^1_{\Sigma}(F, \chi^{-1})$ which is modular by Theorem \ref{BK main}, i.e., we must have $\#\fT_{\rm mod}=1$. The claim now follows directly from Theorem \ref{main3}.  
\end{proof}

\section{$F=\bfQ$} \label{FQ}
In this section we take $F=\bfQ$. 
As in section \ref{The rings Ttau} we set $\tau_1=1$ and $\tau_2=\chi$ where $\chi$ is a character ramified at $p$. By class field theory we can write $\chi = \omega^{k-1} \psi$ for some $k$ with $2 \leq k \leq p-1$ and a character $\psi$ unramified at $p$. We write $\tilde \psi$ for the Teichm\"uller lift of $\psi$ and $\tilde \chi=\tilde \psi \epsilon^{k-1}$.  The assumption that $N \fq \not\equiv 1$ (mod $p$) for all $\fq \in \Sigma$ is unnecessary for any of the results in this section.

\subsection{Proving $\# H^1_{\Sigma}(\bfQ, \tilde{\chi}^{-1}\otimes E/\Oo) \leq  \#\bfT/J$}

For the convenience of the reader let us recall our setup. We denote by $\bfT'$ the Hecke algebra acting on the space of cusp forms $S_k(\Gamma_0(\fN))$ (as before $\fN\in \bfZ_+$ is only divisible by primes in $\Sigma - \{p\}$), i.e., the $\Oo$-subalgebra of $\End_{\bfC}(S_k(\Gamma_0(\fN)))$ generated by $T_{\ell}$ for all $\ell \nmid \fN p$.
Set $J'$ to be the ideal of $\bfT'$ generated by the operators $T_{\ell} - (1+\tilde{\psi}(\ell) \ell^{k-1})$ for all $\ell \not\in \Sigma$.  Let $\fm$ be the maximal ideal of $\bfT'$ containing $J'$ and write $\bfT$ for the completion of $\bfT'$ at $\fm$. Set $J$ to be the image of $J'$ in $\bfT$.

Put $$\eta(\tilde{\psi}, k):= B_k(\tilde{\psi}) \cdot \prod_{\ell \in \Sigma - \{p\}} (1- \tilde{\psi}(\ell) \ell^k),$$ where $B_k(\tilde{\psi})$ is the $k$th Bernoulli number of $\tilde{\psi}$.  
Here we treat $\tilde{\psi}$ as a Dirichlet character of $\bfZ/\fM \bfZ$ rather than of $\bfZ/\fN\bfZ$, where $\fM$ is the largest factor of $\fN$ only divisible by primes dividing the conductor of $\tilde{\psi}$ (in other words we do not set $\tilde{\psi}(\ell)=0$ if $\ell \nmid {\rm cond}(\tilde{\psi})$).

\begin{rem} It is expected that  $\#\bfT/J \geq \#\Oo/\eta(\tilde{\psi}, k)$ as long as $k>2$ or $k=2$ but $\psi \neq 1$. 
The case $k=2$ and $\psi=1$ is slightly different. For $\Sigma=\{p, \ell\}$ with $\ell$ a prime different from $p$ Mazur \cite{Mazur78}  Proposition II.9.7 proved $$\val_p(\#\bfT/J)= [\Oo: \bfZ_p] \val_p({\rm num}\left( \frac{\ell-1}{12}\right)).$$ This corresponds to $\eta(1 \pmod{\ell}, k)$ where we - different to our convention above - take $\tilde \psi=1$ as a Dirichlet character modulo $\ell$, i.e. put $\tilde{\psi}(\ell)=0$.
 In the proof of Proposition \ref{T/J} below the case $k=2$, $\psi=1$ is excluded due to the different form of the constant term of the Eisenstein series. See also \cite{Ohta14} and \cite{Yoo16} who treat a related Hecke algebra when $k=2$, $\psi=1$ and the level is composite. 
\end{rem}
We now prove that  $\#\bfT/J \geq \#\Oo/\eta(\tilde{\psi}, k)$ under some conditions.

\begin{prop} \label{T/J}
Let $k \geq 2$. If $k=2$ assume that $\psi \neq1$. Let $N={\rm cond}(\tilde{\psi})$, $\Sigma=\{p, \ell, q \mid N\}$ for some prime $\ell \nmid Np$.   Then  there exists $m>0$ such that  $\#\bfT/J \geq \#\Oo/\eta(\tilde{\psi}, k)$ for $\fN=N\ell^m$.
\end{prop}
\begin{rem}
We note that our proof in fact shows that  $\#\tilde \bfT/\tilde J \geq \#\Oo/\eta(\tilde{\psi}, k)$, where $\tilde \bfT$ is the Hecke algebra including $T_p$, and $\tilde J$ has the additional generator $T_p-(1+\tilde \psi(p)p^{k-1})$. Note that  $\bfT/J \twoheadrightarrow \tilde \bfT/\tilde J$. We do not use the congruence module $ \tilde \bfT/\tilde J$ in this paper, but for other applications it might be of interest that the corresponding cusp forms congruent to the Eisenstein series are ordinary at $p$. Let us also note that for Proposition \ref{T/J} we allow for the primes dividing $\fN$ to be congruent to 1 mod $p$.
\end{rem}

\begin{proof} [Proof of Proposition \ref{T/J}] We partially  adapt arguments from lectures notes by Skinner from 2002 which treat the case of weight $k=2$ (making explicit Wiles' argument in the proof of the totally real Iwasawa Main Conjecture).

If $\eta(\tilde{\psi}, k) \in \Oo^{\times}$ then there is nothing to prove. So assume ${\rm val}_{\varpi}(\eta(\tilde{\psi}, k))>0$.
Let $\phi$ be a non-trivial Dirichlet character of conductor $M$ such that $\phi(-1) = (-1)^l$ for $l\geq 1$. Set $$E_l(\phi) = \frac{L(\phi, 1-l)}{2} + \sum_{n=1}^{\infty} \left( \sum_{d \mid n}\phi(d)d^{l-1}\right) q^n   \in M_l(M, \phi)$$ to be the Eisenstein series of weight $l$ whose constant term is $L(\phi, 1-l)/2$ (cf.~\cite{Miyake89}, Theorem 4.7.1). 
\begin{prop}[\cite{Ozawa17} Proposition 0.3]
If $l=2$ assume that $\phi \neq1$. The constant term of $E_l(\phi)$ at the cusp $[u:v] \in \bfP^1(\bfQ)$ equals $\phi(u)^{-1}L(\phi, 1-l)/2$ if $M \mid v$ and zero otherwise. 
\end{prop}

By a generalisation of a result of Washington (see \cite{Sun10} Theorem 4) we know that there exists an auxiliary character $\varphi$ of conductor $\ell^m$ for some $m>0$ (which we fix from now on) with $\varphi(-1)=(-1)^{k-1}$ such that \be \label{punit} L(\tilde \psi \varphi, 0) L(\varphi^{-1}, 2-k) \in \Oo^{\times}.\ee 

Then we put
$$G:=E_1( \tilde \psi \varphi) \cdot E_{k-1}( \varphi^{-1}) \in M_k(N\ell^m, \tilde \psi)$$ and deduce that its constant terms are
$$\begin{cases}
\tilde \psi^{-1}(u) \frac{L(\tilde \psi \varphi, 0) L(\varphi^{-1}, 2-k)}{4}& \text{if } \, N \ell^m \mid v\\
0& \text{else}.
\end{cases}$$

In the following we will use $G$, which clearly has $p$-integral Fourier coefficients and a constant term which is a $p$-unit, to prove a congruence of the following Eisenstein series to a cusp form.
Put
$$F_m(z):=E_k(\tilde \psi)(\ell^{m-1}z) - \tilde{\psi}(\ell)\ell^k E_k(\tilde \psi)(\ell^m z).$$ We apply Proposition 1.2 in \cite{BillereyMenares16} (generalized to $k \geq 2$ (and $\psi\neq 1$ if $k=2$) in \cite{BillereyMenares18} Proposition 4) with $M=\ell^{m-1}$ (for $E_k(\tilde \psi)(\ell^{m-1}z)$) and $M=\ell^m$ (for $E_k(\ell^m z)$) to compute that the constant term of $F_m$ at the cusp $[u:v]$ equals $$-\tilde{\psi}(u)^{-1}\frac{B_{k, \tilde{\psi}}}{2k}(1-\tilde{\psi}(\ell)\ell^k) = \tilde{\psi}(u)^{-1}\frac{L(\tilde{\psi}, 1-k)}{2}(1-\tilde{\psi}(\ell)\ell^k)$$ if $N\ell^m \mid v$ and zero otherwise.

This now allows us to get a bound on $\bfT/J$: 
Define $$H=F_m- \frac{\eta(\tilde{\psi}, k)}{a_0(G)k} \cdot G,$$ where $a_0(G)$ denotes the constant term of $G$ at infinity (which is a $p$-unit - see above).
Then the previous discussion shows that $H \in S_k(N\ell^m, \tilde \psi)$ with $q$-expansion coefficients in $\Oo$.

We can then define a surjective $\Oo$-algebra homomorphism $\phi: \bfT/J \twoheadrightarrow \Oo/\eta(\tilde{\psi}, k)$
 such that $T_q \mapsto 1+ \tilde \psi(q) q^{k-1}$ for all primes $q \nmid N \ell p$ as follows: 

First note that $H$ has a Fourier coefficient which is a $p$-unit. To see this, note $a_{\ell^{m-1}}(F_m)=a_1(E_k(\tilde \psi))=1$, so $$a_{\ell^{m-1}}(H)=1-\frac{\eta(\tilde{\psi}, k)}{a_0(G)k}  \cdot a_{\ell^{m-1}}(G) \in \Oo^{\times},$$ where $a_n$ denotes the $n$-th Fourier coefficient of the respective modular form. 

This allows us to extend $H$ to an $\Oo$-basis of $S_k(\ell^m N, \Oo)$, say $m_0=H, m_1, \ldots m_r$. Let $t \in \bfT$. Then $$t m_0=\sum_{i=0}^r \lambda_i(t) m_i, \, \text{ for } \lambda_i(t) \in \Oo.$$ 
We can now define the (surjective) $\Oo$-module homomorphism $\phi: \bfT\to \Oo/\eta(\tilde{\psi},k)$ by $\phi(t)=\lambda_0(t)$ (mod $\eta(\tilde{\psi},k))$, and it is easy to check that this, in fact, is even a ring homomorphism, and that it factors through $\bfT/J$ since $T_q-1-\tilde \psi(q) q^{k-1}$ annihilates $F_m$.
\end{proof}

\begin{rem}
Dummigan-Fretwell \cite{DummiganFretwell14}, Billerey-Menares \cite{BillereyMenares18}, and Spencer \cite{Spencer18} use similar linear combinations of Eisenstein series to prove mod $p$  congruences using the Deligne-Serre lifting lemma. Note, however, that our $F_m$ has non-vanishing constant terms only for $N \ell^m \mid v$, which makes it possible to remove them by using the auxiliary $G$ and prove the full expected $\bfT/J$ bound. By \cite{BergerKlosinKramer14} Proposition 4.3 this gives a lower bound on the amount and depth of Eisenstein congruences:

For a Hecke eigensystem $\lambda: \bfT \to \Oo$ write $m_{\lambda}$ for the depth of its  $p$-adic congruence with $E_k(\tilde \psi)$, i.e., $m_{\lambda}$ is the largest integer $s$ such that $\lambda(T_{\ell})\equiv 1+\tilde{\psi}(\ell)\ell^{k-1}$ mod $\varpi^s$ for every $\ell \not\in \Sigma$. Write $e$ for the ramification index of $\Oo$ over $\bfZ_p$. Then combining Proposition\ref{T/J} with \cite{BergerKlosinKramer14} Proposition 4.3  we obtain $$\frac{1}{e} \sum_{\lambda} m_{\lambda} \geq \val_p(\#\bfT/J) \geq \val_p(\#\Oo/\eta(\tilde{\psi}, k)).$$ 
\end{rem}

\begin{prop} \label{upper bound 1} One has $\# H^1_{\Sigma}(\bfQ, \tilde{\chi}^{-1} \otimes E/\Oo) \leq \#\Oo/\eta(\tilde{\psi}, k)$. \end{prop}
\begin{proof} Consider the following diagram of fields with corresponding Galois groups:
$$
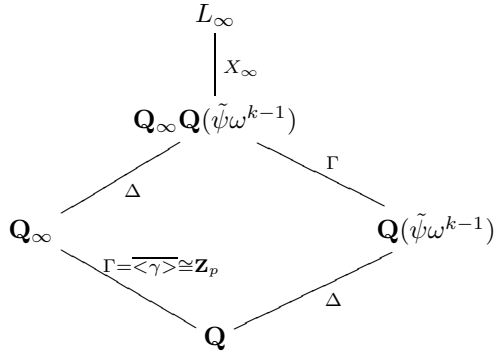
\xymatrix{&L_{\infty}\ar@{-}[d]^{X_{\infty}}\\ &\bfQ_{\infty}\bfQ(\tilde{\psi}\omega^{k-1})\ar@{-}[dl]^{\Delta}\ar@{-}[dr]^{\Gamma}\\
\bfQ_{\infty}\ar@{-}[dr]^{\Gamma=\ov{<\gamma>}\cong \bfZ_p}&&\bfQ(\tilde{\psi}\omega^{k-1})\ar@{-}[dl]^{\Delta}\\
&\bfQ
}$$ Here $\bfQ(\tilde{\psi}\omega^{k-1})$ denotes the splitting field of $\tilde{\psi}\omega^{k-1}$ and $L_{\infty}$ is the maximal abelian extension of $\bfQ_{\infty}\bfQ(\tilde{\psi}\omega^{k-1})$ unramified everywhere.

We first prove that \be \label{MCupper} \# H^1_{\{p\}}(\bfQ, \tilde{\chi}^{-1} \otimes E/\Oo) \leq \#\Oo/B_k(\tilde{\psi}).\ee This follows from the Main Conjecture of Iwasawa theory proven my Mazur-Wiles, as we briefly explain for the convenience of the reader:
 For $K=\bfQ$ or $\bfQ_{\infty}$ and $\varphi$ a character of $G_K$ put $$H^1_{\rm Gr}(K, E/\Oo(\varphi)):=\ker(H^1(K,E/\Oo(\varphi)) \to \prod_{v} H^1(I_v, E/\Oo(\varphi))).$$
A result of Flach (see \cite{Ochiai00} Proposition 4.1(1)) tells us that $$H^1_{\{p\}}(\bfQ,  E/\Oo(\tilde \psi^{-1} \epsilon^{1-k})) \subseteq H^1_{\rm Gr}(\bfQ, E/\Oo(\tilde \psi^{-1} \epsilon^{1-k})).$$
Let $\Psi=\tilde \psi^{-1} \omega^{1-k}$ and $X_{\infty, \Psi}$ be the $\Psi$-isotypical component of $X_{\infty}$ for the action of $\Delta$. 
We have $X_{\infty, \Psi}=\Hom(H^1_{\rm Gr}(\bfQ_{\infty}, E/\Oo(\Psi)), E/\Oo).$
Using the $\Gamma$-module structure of $X_{\infty, \Psi}$ from this we get $$X_{\infty, \Psi}/(T-(\kappa_0^{1-k}-1))=\Hom(H^1_{\rm Gr}(\bfQ, E/\Oo(\Psi (\epsilon/\omega)^{1-k})), E/\Oo),$$ where $\kappa_0=(\epsilon/\omega)(\gamma)$. Since both modules are finite and $\Psi (\epsilon/\omega)^{1-k}=\tilde \psi^{-1} \epsilon^{1-k}$ we get 
$$\#H^1_{\rm Gr}(\bfQ, E/\Oo(\tilde \psi^{-1} \epsilon^{1-k}))=\#X_{\infty, \Psi}/(T-(\kappa_0^{1-k}-1)).$$
Since $X_{\infty, \Psi}$ has no finite $\Lambda:=\bfZ_p[[\Gamma]]$-submodules (see \cite{MazurWiles84} Proposition 1 on p. 193) one obtains
$$\#X_{\infty, \Psi}/(T-(\kappa_0^{1-k}-1)) \leq \# \Lambda/(g_{\Psi}, T-(\kappa_0^{1-k}-1)), $$ where $g_{\Psi} \in \Lambda$ is the characteristic power series of $X_{\infty, \Psi}$.
By the Main Conjecture (see \cite{MazurWiles84} Theorem p. 214) we have $$g_{\Psi}(\kappa_0^s-1)=L_p(\omega \Psi^{-1}, s), $$ where the latter is the $p$-adic $L$-function with the following interpolation property (see \cite{Washingtonbook} Theorem 5.11):  $$L_p(\omega \Psi^{-1}, 1-n)=-(1-\tilde \psi(p)p^{n-1}) \frac{B_n(\tilde \psi)}{n},  \text{ for } n \geq 1.$$
Setting $n=k$ and observing that $(1-\tilde \psi(p)p^{k-1}) \in \Oo^{\times}$ we obtain \eqref{MCupper}.

A repeated application of Lemma \ref{lower bound 21} in the next section (by selecting $s$ in that lemma to be sufficiently large and taking $n$ in that lemma to be $k-1$) leads us now to the bound by $\eta(\tilde{\psi}, k)$ on $H^1_{\Sigma}(\bfQ, \tilde{\chi}^{-1}\otimes E/\Oo)$.
\end{proof}

From now on assume that $\tilde{\psi}$, $\Sigma$ and $\bfT$ are as in Proposition \ref{T/J}.  
By combining Propositions \ref{3rdinequality}, \ref{T/J} and \ref{upper bound 1} we obtain the following corollary.
\begin{cor} \label{allequal}
We have $$\#\bfT/J = \#\Oo/\eta(\tilde{\psi}, k)=\# H^1_{\Sigma}(\bfQ, \tilde{\chi}^{-1} \otimes E/\Oo).$$
\end{cor}

Then in the case $F=\bfQ$ we obtain the following stronger versions of Theorem \ref{main3} and Corollary \ref{main4}. 
\begin{cor} \label{sum up 2} If $\dim_{\bfF} H^1_{\Sigma}(\bfQ, \chi)=1$ and the Eisenstein ideal $J$ is not principal, then $\#\fT_{\rm mod} > \dim_{\bfF}H^1_{\Sigma}(\bfQ, \chi^{-1})$. \end{cor}

\begin{rem} Suppose we consider the set of extensions $\ov{\rho}=\bmat 1&*\\ & \chi\emat: G_{\Sigma'}\to\GL_2(\bfF)$ with $\chi$ ramified at all primes in $\Sigma'\supset\{p\}$. Then Corollary \ref{sum up 2} can be viewed as asserting that more than $\dim_{\bfF}H^1_{\Sigma}(\bfQ, \chi^{-1})$ of these extensions arise from modular representations $\rho_f$ which are ramified at no more than one additional prime (the prime $\ell$ in Proposition \ref{T/J}, i.e., $\Sigma = \Sigma'\cup\{\ell\}$) as long as $J$ is not principal and  $\dim_{\bfF} H^1_{\Sigma}(\bfQ, \chi)=1$.
\end{rem}
 
\begin{cor} \label{main2}  Suppose that $\dim_{\bfF} H^1_{\Sigma}(\bfQ, \chi)=\dim_{\bfF} H^1_{\Sigma}(\bfQ, \chi^{-1})=1$. Then $J$ is principal. \end{cor}

\subsection{Congruence criterion} The assumption that the Eisenstein ideal is not principal may be difficult to check directly, so we will translate it here into a criterion that relies on counting congruences. We still let $\tilde{\psi}$, $\Sigma$ and $\bfT$ be as in Proposition \ref{T/J}. 

For a Hecke eigensystem $\lambda: \bfT \to \Oo$ write $m_{\lambda}$ for the depth of its  $p$-adic congruence with $E_k(\tilde \psi)$, i.e., $m_{\lambda}$ is the largest integer $s$ such that $\lambda(T_{\ell})\equiv 1+\tilde{\psi}(\ell)\ell^{k-1}$ mod $\varpi^s$ for every $\ell \not\in \Sigma$. Write $e$ for the ramification index of $\Oo$ over $\bfZ_p$. 

\begin{thm} \label{main} Assume that $\dim_{\bfF}H^1_{\Sigma}(\bfQ, \chi)=1$. If $$\frac{1}{e} \sum_{\lambda} m_{\lambda} > \val_p(\#\Oo/\eta(\tilde{\psi}, k))$$ then $J$ is not principal and $\#\fT_{\rm mod} > \dim_{\bfF}H^1_{\Sigma}(\bfQ, \chi^{-1})$. \end{thm}
\begin{proof} 
Assume $J$ is principal. Writing $T_{\lambda}=\Oo$, $J_{\lambda}=\varpi^{m_{\lambda}}\Oo$, $T=\bfT$ and $J$ as before for the Eisenstein ideal, we can apply Corollary     2.7 in \cite{BergerKlosinKramer14} (again note as in Proposition \ref{main1} that the principality of the $J_{\lambda}$s) to conclude that  then $$\val_{\varpi}(\#T/J)=\val_{\varpi}\left(\#\prod_{\lambda}T_{\lambda}/J_{\lambda}\right) = \frac{[E:\bfQ_p]}{e}\sum_{\lambda} m_{\lambda}.$$
The left-hand side  equals $\val_{\varpi}(\#\Oo/\eta(\tilde{\psi}, k))$ by Corollary \ref{allequal}. Replacing $\varpi$-adic valuations with $p$-adic ones we get $\frac{1}{e} \sum_{\lambda} m_{\lambda} = \val_p(\#\Oo/\eta(\tilde{\psi}, k))$, which contradicts our assumption. So we conclude that $J$ is not principal and the proposition follows by applying Proposition \ref{main1}. \end{proof}

\section{Analysis of $H^1_{\Sigma}(\bfQ, \bfF(n))$}\label{Analysis}

In this section we prove bounds on certain Selmer groups. The assumption that $\ell \not\equiv 1$ (mod $p$) for all $\ell \mid N$ is not needed for Proposition \ref{prop1-k} and Lemma \ref{lower bound 21}.
\begin{prop} \label{prop1-k}For $2 \leq k \leq p-1$  and $k$ even we have $${\rm val}_p(\#H^1_{\Sigma}(\bfQ, \bfF(1-k)))\geq [\bfF:\bfF_p]( \min \{{\rm val}_p(B_{1, \omega^{k-1}}), 1\} +  \sum_{\ell \in \Sigma - \{p\}} \min \{\val_p(1- \ell^k), 1\}).$$
\end{prop}
\begin{proof}
By Fontaine-Laffaille theory (see e.g. \cite{Breuil01} Proposition 9.1.2(i)) 
any Fontaine-Laffaille $D_p$ extension $$0 \to \bfF \to \ov{\rho} \to \bfF(k-1) \to 0$$ is split on $I_p$, so $H^1_f(\bfQ_p, \bfF(1-k))=H^1_{\rm ur}(\bfQ_p, \bfF(1-k)):=(\ker(H^1(\bfQ_p, \bfF(1-k)) \to H^1(I_p, \bfF(1-k)))$.
 We therefore have $$H^1_{\{p\}}(\bfQ,  \bfF(1-k))={\rm ker}\left(H^1(\bfQ,  \bfF(1-k)) \to \prod_{\ell}H^1(I_{\ell}, \bfF(1-k))\right).$$ As  in section 2 of \cite{Skinner06} we can argue that restriction to $G_{\bfQ(\mu_p)}$ gives $$H^1_{\{p\}}(\bfQ,  \bfF(1-k))=\Hom_{\Gal(\bfQ(\mu_p)/\bfQ)}(C_{\bfQ(\mu_p)}, \bfF(1-k)),$$ where $C_{\bfQ(\mu_p)}$ denotes the class group of $\bfQ(\mu_p)$.

The $p$-primary part of $C_{\bfQ(\mu_p)}$ on which the action of $\Gal(\bfQ(\mu_p)/\bfQ)$ is via $\omega^{1-k}$  has order given by $L(0, \omega^{k-1})=-B_{1, \omega^{k-1}}$ by \cite{MazurWiles84} Theorem 2 p. 216 (see also \cite{Skinner06} Theorem 2.1.3).
This shows that $\#H^1_{\{p\}}(\bfQ, \bfF(1-k)) \geq ( \#\bfF_p/B_{1, \omega^{k-1}})^{ [\bfF:\bfF_p] }$ (equality holds if $C_{\bfQ(\mu_p)}^{\omega^{1-k}}$ is cyclic). 

The proposition now follows from Lemma \ref{lower bound 21} below applied with $n=k-1$.
\end{proof}

\begin{lemma} \label{lower bound 21}  Let $n \neq 0$ be an integer and set $m:=\val_p(\tilde \psi(\ell) \ell^{n+1} -1)$ for $\psi$ a Dirichlet character unramified away from $\Sigma-\{p\}$. Let $s\geq me$ be an integer, where $e$ is the ramification index of $\Oo$ over $\bfZ_p$. Set $W=E/\Oo(\tilde\psi^{-1} \epsilon^{-n})$ and $W_s=W[\varpi^s]$. Suppose $\ell \in \Sigma - \{p\}$ and let $\Sigma' \subset \Sigma$ with $\ell \not\in \Sigma'$.  Then one has $$\#H^1_{\Sigma' \cup \{\ell\}}(\bfQ,W_s) \leq  (\# \Oo/p^m\Oo) \#H^1_{\Sigma'}(\bfQ, W_s).$$
\end{lemma}

\begin{proof}   First assume that $W$ is ramified  at $\ell$. Then $W^{I_{\ell}}=0$ and we use \cite{BergerKlosin13} Lemma 5.6 to conclude that $$H^1_{\Sigma' \cup \{\ell\}}(\bfQ,W_s)=H^1_{\Sigma'}(\bfQ,W_s).$$ From now on assume that $W$ is unramified at $\ell$. By  \cite{Rubin00}, Theorem 1.7.3 we have an exact sequence $$0 \to H^1_{\Sigma'}(\bfQ, W_s) \to H^1_{\Sigma' \cup \{\ell\}}(\bfQ, W_s) \to \frac{H^1(\bfQ_{\ell}, W_s)}{H^1_{\rm ur}(\bfQ_{\ell}, W_s)}.$$ Lemma 1.3.8(ii) in \cite{Rubin00} tells us that $H^1_{\rm ur}(\bfQ_{\ell}, W_s) = H^1_f(\bfQ_{\ell}, W_s)$, where $H^1_{\rm un}(\bfQ_{\ell}, W_s):= \ker (H^1(\bfQ_{\ell}, W_s) \to H^1(I_{\ell}, W_s))$. We also get \be\label{eq0}H^1(I_{\ell}, W_s) = \Hom(I_{\ell}^{\rm ab}, W_s) = \Hom(\bfZ_p(1), W_s) =W_s(-1).\ee This gives an upper bound of $(\#\bfF)^s=\# W_s$ on the order of the quotient $\frac{H^1(\bfQ_{\ell}, W_s)}{H^1_{\rm ur}(\bfQ_{\ell}, W_s)}$. To prove the claim it is enough to show that the image of the map $H^1(\bfQ_{\ell}, W_s) \to H^1(I_{\ell}, W_s)$ has order not greater than $\#\Oo/p^m \Oo$.  To do so consider the inflation-restriction sequence (where we set $G:=\Gal(\bfQ_{\ell}^{\rm ur}/\bfQ_{\ell})$):
$$H^1(G, W_s) \to H^1(\bfQ_{\ell}, W_s) \to H^1(I_{\ell}, W_s)^{G}\to H^2(G, W_s).$$ The last group in the above sequence is zero since $G\cong \hat{\bfZ}$ and $\hat{\bfZ}$ has cohomological dimension one. This means that the image of the restriction map $H^1(\bfQ_{\ell}, W_s) \to H^1(I_{\ell}, W_s)$ equals $H^1(I_{\ell}, W_s)^{G}$. Let us show that the  latter module has order $\leq \#\Oo/p^m \Oo$. Indeed, \begin{multline} H^1(I_{\ell}, W_s)^{G}=\Hom_G(I_{\ell}, W_s) = \Hom_G(I_{\ell}^{\rm tame}, W_s) \\= \Hom_G(\bfZ_p(1), p^{-s}\Oo/\Oo(\tilde \psi^{-1} \epsilon^{-n})) = \Hom_G(\bfZ_p, p^{-s}\Oo/\Oo(\tilde \psi^{-1} \epsilon^{-n-1})).\end{multline} So, $\phi \in H^1(I_{\ell}, W_s)$ lies in $H^1(I_{\ell}, W_s)^G=\Hom_G(\bfZ_p, p^{-s}\Oo/\Oo(\tilde \psi^{-1} \epsilon^{-n-1}))$ if and only if $\phi(x)=g \cdot \phi(g^{-1}\cdot x) = g\cdot \phi(x)=\tilde \psi^{-1} \epsilon^{-n-1}(g)\phi(x)$ for every $x \in I_{\ell}$ and every $g \in G$, i.e., if and only if \be \label{eq1} (\tilde \psi^{-1} \epsilon^{-n-1}(g)-1)\phi(x)\in \Oo\quad \textup{for every $x \in I_{\ell}$, $g \in G$.}\ee Since $\Frob_{\ell}$ topologically generates $G$, we see that \eqref{eq1} holds if and only if it holds for every $x \in I_{\ell}$ and for $g=\Frob_{\ell}$.  So condition \eqref{eq1} becomes \be \label{eq2} (1-\tilde \psi^{-1}(\ell)\ell^{-n-1})\phi(x)\in \Oo\quad \textup{for every $x \in I_{\ell}$.}\ee Since $\val_p(1-\tilde \psi^{-1}(\ell)\ell^{-n-1})= \val_p(\tilde \psi(\ell) \ell^{n+1}-1)=m$, we get that $\phi(x) \in p^{-m}\Oo/\Oo$, as claimed.
\end{proof}

When $s=1$ and $\psi=1$ we prove a stronger result.

\begin{lemma} \label{lower bound 11}  Let $n \neq 0$ be an integer. Suppose $\ell \in \Sigma - \{p\}$ and $m:=\min\{\val_p(\ell^{n+1} -1),1\}$. Let $\Sigma' \subset \Sigma$ with $\ell \not\in \Sigma'$. Write $q = \#\bfF$. Then one has $$\#H^1_{\Sigma' \cup \{\ell\}}(\bfQ, \bfF(-n)) =  q^m \#H^1_{\Sigma'}(\bfQ, \bfF(-n)).$$
\end{lemma}
\begin{proof} If $m=0$ the inequality \be \label{one ineq} \#H^1_{\Sigma' \cup \{\ell\}}(\bfQ, \bfF(-n)) \leq   q^m \#H^1_{\Sigma'}(\bfQ, \bfF(-n))\ee follows directly from Lemma \ref{lower bound 21}, while the opposite inequality is clear.  

As before, set $W=E/\Oo(-n)$ and $W_s=W[\varpi^s]$. Then for $m=1$, \cite{Rubin00}, Theorem 1.7.3 gives us again an exact sequence \be \label{pt} 0 \to H^1_{\Sigma'}(\bfQ, W_1) \to H^1_{\Sigma' \cup \{\ell\}}(\bfQ, W_1) \to \frac{H^1(\bfQ_{\ell}, W_1)}{H^1_{\rm ur}(\bfQ_{\ell}, W_1)},\ee and as in the proof of Lemma \ref{lower bound 21} we see that the oder of the module on the right is bounded by $q$. This yields \eqref{one ineq}.   

We now show that the third arrow (which we call ${\rm loc}^s$ following \cite{Rubin00}, section 1.7) in \eqref{pt} is surjective if $\val_p(\ell^{n+1}-1)>0$ and is the zero-map otherwise.

As, before, since $W$ is unramified at $\ell$, Lemma 1.3.5(iv) in \cite{Rubin00} implies that $H^1_{\rm ur}(\bfQ_{\ell}, W) = H^1_f(\bfQ_{\ell}, W)$, where $H^1_{\rm un}(\bfQ_{\ell}, W):= \ker (H^1(\bfQ_{\ell}, W) \to H^1(I_{\ell}, W))$. Similarly, this time using Lemma 1.3.8(ii) in \cite{Rubin00} we get that $H^1_{\rm ur}(\bfQ_{\ell}, W_1) = H^1_f(\bfQ_{\ell}, W_1)$, where $H^1_{\rm un}(\bfQ_{\ell}, W_1):= \ker (H^1(\bfQ_{\ell}, W_1) \to H^1(I_{\ell}, W_1))$.

Write $\mS_{\Sigma'}(\bfQ, W_1^*)$ for the kernel of the map $H^1_{\Sigma'}(\bfQ, W_1^*) \to \bigoplus_{v \in \Sigma'} H^1(\bfQ_v, W_1^*)$ (cf.~\cite{Rubin00}, p.21-22) and analogously for $\mS_{\Sigma' \cup \{\ell\}}$. Here $W_1^* = \Hom(W_1, \bfF)(1) = \bfF(n+1)$.  The cup product induces a perfect pairing $H^1(\bfQ_v, W_1) \times H^1(\bfQ_v, W_1^*) \to H^2(\bfQ_v, \bfF(1)) \cong \bfF(1).$ Theorem 1.7.3(ii) in \cite{Rubin00} yields an exact sequence \be \label{pt2} 0 \to \mS_{\Sigma'\cup \{\ell\}}(\bfQ, W_1^*) \to \mS_{\Sigma'}(\bfQ, W_1^*) \xrightarrow{{\rm loc}_f} H^1_f(\bfQ_{\ell}, W_1^*)\ee where the last module again equals $H^1_{\rm ur}(\bfQ_{\ell}, W_1^*)$ as above.  By the same theorem the image of ${\rm loc}^s$ is the largest subspace of $ \frac{H^1(\bfQ_{\ell}, W_1)}{H^1_{\rm ur}(\bfQ_{\ell}, W_1)}$ having the property that all of its elements pair to zero with any element of the image of ${\rm loc}_f$. Thus to show surjectivity of ${\rm loc}^s$, it is enough to show that ${\rm loc}_f$ is the zero map, i.e., that $\mS_{\Sigma'\cup \{\ell\}}(\bfQ, W_1^*) = \mS_{\Sigma'}(\bfQ, W_1^*)$. Consider the inclusion $\mS_{\Sigma'\cup \{\ell\}}(\bfQ, W_1^*) \subset \mS_{\Sigma'}(\bfQ, W_1^*)$ and assume that $\phi \in \mS_{\Sigma'}(\bfQ, W_1^*)$. This in particular means that $\phi |_{G_{\bfQ_{\ell}}} \in H^1_{\rm ur}(\bfQ_{\ell}, W^*_1)$. If we can show that this forces $\phi|_{G_{\bfQ_{\ell}}}$ to be zero, then we get $\phi \in  \mS_{\Sigma'\cup \{\ell\}}(\bfQ, W_1^*)$ as desired. This will follow if we show that $ \frac{H^1(\bfQ_{\ell}, W_1^*)}{H^1_{\rm ur}(\bfQ_{\ell}, W_1^*)}=0$, i.e., that $H^1_{\Sigma'}(\bfQ, W_1^*) = H^1_{\Sigma'\cup \{\ell\}}(\bfQ, W_1^*)$. Clearly, all we need is  $\#  H^1_{\Sigma'\cup \{\ell\}}(\bfQ, W_1^*) \leq \# H^1_{\Sigma'}(\bfQ, W_1^*)$, which will follow from \eqref{one ineq} applied to $W_1^*$, i.e., replacing $-n$ by $n+1$ as long as we can show that the corresponding value of $m$, which for $W_1^*$ will be $\min \{\val_p(\ell^{-n}-1), 1\}$ is zero. This follows if we show $\val_p(\ell^n-1)=0$. Suppose that $\ell^n \equiv 1$ (mod $p$). Then by our assumption that $\val_p(\ell^{n+1}-1)>0$ we get $1 \equiv \ell^{n+1} \equiv \ell$ (mod $p$) which contradicts the assumption that $\ell \not\equiv 1$ mod $p$. This completes the proof.  
\end{proof}

 For $H^1_{\Sigma}(\bfQ, \bfF(k-1))$ on the other hand it is in general not possible to relate to pieces of class groups, as $H^1_f(\bfQ_p, \bfF(k-1)) \neq H^1_{\rm ur}(\bfQ_p, \bfF(k-1))$ (but see \cite{Rubin00} Proposition 1.6.4(ii) for $k=1$).
 
\begin{prop} \label{propk-1a}One has 
 \begin{multline}  {\rm val}_p(\#H^1_{\Sigma}(\bfQ, \bfF(k-1)))\leq{\rm val}_p(\#H^1(\Gal(\bfQ_{\{p\}}/\bfQ), \bfF(k-1))) \\+ [\bfF:\bfF_p]  \sum_{\ell \in \Sigma - \{p\}} \min \{\val_p(1- \ell^{k-2}), 1\}.\end{multline} \end{prop}

\begin{proof} Let us first assume that one has 
 \begin{multline} \label{first step} {\rm val}_p(\#H^1_{\Sigma}(\bfQ, \bfF(k-1)))\leq {\rm val}_p(\#H^1_{\{p\}}(\bfQ, \bfF(k-1))) \\+   [\bfF:\bfF_p]\sum_{\ell \in \Sigma - \{p\}} \min \{\val_p(1- \ell^{k-2}), 1\}.\end{multline}
The Selmer group $H^1_{\{p\}}(\bfQ, \bfF(k-1))$ is certainly no larger than the Selmer group where all the classes are unramified away from $p$ and we impose no condition at $p$. This last Selmer group is isomorphic to $H^1(\Gal(\bfQ_{\{p\}}/\bfQ), \bfF(k-1))$. Here $\bfQ_{\{p\}}$ stands for the maximal algebraic extension of $\bfQ$ unramified away from $p$. This gives us the claim of the Proposition. Hence it remains to prove \eqref{first step}, but this follows by (a possibly repeated application of) Lemma \ref{lower bound 11} where we set $n=1-k$ and note that $\val_p(\ell^{k-2}-1) = \val_p(\ell^{2-k}-1)$. \end{proof}

We will use the following proposition with $r=k-1$. 
\begin{prop} \label{propk-1b} Suppose $r\in \bfZ$, $r> 1$ and that the $\epsilon^{r}$-eigenspace of the $p$-part $C$ of the class group of $\bfQ(\mu_p)$ is trivial. Then $\dim_{\bfF}H^1(\Gal(\bfQ_{\{p\}}/\bfQ), \bfF(r)) \leq 1$. 
\end{prop} 
\begin{proof} Write $G$ for $\Gal(\bfQ_{\{p\}}/\bfQ)$. Using the inflation-restriction sequence we need to show that $$\dim_{\bfF} \Hom_G((\ker \ov{\epsilon}^r)^{\rm ab}, \bfF(r))\leq 1.$$ By Class Field Theory this reduces the problem to studying the units for the splitting field of $\chi_0:=\ov{\epsilon}^{r}$ as a $\Gal(\bfQ(\chi_0)/\bfQ)$-module. A similar analysis has been carried in section 3 of \cite{BergerKlosin09} for imaginary quadratic fields. The current situation is simpler, so we will only sketch the argument here and refer the reader to \cite{BergerKlosin09} for details. Write $M$ for the group of local (at $p$ - note that $p$ ramifies totally in $\bfQ(\chi_0)$) units of $\bfQ(\chi_0)$ and $T$ for its torsion subgroup.
  Then $M/T$ is a free $\bfZ_p$-module of rank $d:=[\bfQ(\chi_0):\bfQ]$. Since the $\ov{\epsilon}^r$-eigenspace of $C$ is trivial, by Proposition 13.6 in \cite{Washingtonbook} we see that any element of $\Hom_G((\ker \chi_0)^{\rm ab}, \bfF(r))$ gives rise to a $G$-equivariant homomorphism from $M$ to $\bfF(r)$. As $T\cong \mu_p$ and so $G$ acts on $T$ by $\ov{\epsilon}$ we see that such a homomorphism will factor through $M/T$ as $r \neq 1$. Using $M/T \cong 1+\fP$, where $\fP$ is the prime of $\bfQ(\chi_0)$ lying over $p$, it is enough to  decompose $\fP$ as a $G$-module. One easily sees that $\fP=\bigoplus_{i=0}^{p-2}\bfF(\ov{\epsilon}^{i})$. 
\end{proof}

\section{Example}\label{Ex} 
We end with an example, where the conditions of Theorem \ref{main} are satisfied.

Let $p=37$, $k=32$, $\Sigma=\{31, 37\}$, and consider $\chi=\omega^{k-1}$ (i.e. $\psi=1$).  Since $p \nmid (1-31^{30})$ we have by Lemma \ref{lower bound 11} that $H^1_{\Sigma}(\bfQ, \chi)=H^1_{\{p\}}(\bfQ, \chi)$. 
By Propositions  \ref{propk-1a} and \ref{propk-1b} we know that the latter is at most 1-dimensional since the relevant piece of the class group of $\bfQ(\mu_p)$ is trivial as $p \nmid B_6$ by Herbrand's theorem. 
Using MAGMA \cite{MAGMA97} one confirms that there are cuspforms of weight 32 of level 1 congruent to Eisenstein series, so by Ribet's lattice construction we know that there exists a non-trivial crystalline extension $\begin{pmatrix} \chi&*\\0&1\end{pmatrix}$, so $\dim_{\bfF} H^1_{\Sigma}(\bfQ, \chi)=1$.

While our arguments below (together with Theorem \ref{main})  imply in particular that $\dim_{\bfF}H^1_{\Sigma}(\bfQ, \chi^{-1})\geq 2$ (so the question of the number of modular extensions becomes relevant) we note that this also follows from Proposition \ref{prop1-k} since $p \mid B_{32}$ (which by the Kummer congruences implies $p \mid B_{1, \omega^{31}}$) and $p \mid (1-31^{32})$.

Since $\eta(\mathbf{1}, 32)=B_{32} (1-31^{32})$ has $\val_{37}=2$ Proposition \ref{T/J} implies that  $\#\bfT/J \geq \# \Oo/p^2$ for $\bfT$ the completion of the Hecke algebra acting on $S_{32}(\Gamma_0(31))$, as one can check using SAGE \cite{sagemath} that there exists a character of conductor $31$  satisfying \eqref{punit} (so $m=1$ in the statement of Proposition \ref{T/J}).

MAGMA calculations further show that  $S_{32}(\Gamma_0(31))$ has 2 Galois conjugacy classes of newforms. One of these has a coefficient field of degree 37 over $\bfQ$. We were not able to calculate its integer ring, but we could check that 37 factors over this field as $\mathfrak{P}_1\mathfrak{P}_2\mathfrak{P}_3\mathfrak{P}_4\mathfrak{P}_5\mathfrak{P}_6$, where only $\mathfrak{P}_1$ and $\mathfrak{P}_2$ have inertia degree 1. Using MAGMA we calculated   the absolute norm of $(a_n(f)-(1+n^{31})) \mod{37}$ for the newforms $f \in S_{32}(\Gamma_0(31))$ and $n=2,3,5$. This gives zero for all 37 Galois conjugates, but not zero modulo $37^2$. This means that all 37 conjugates in the first class are congruent  to the Eisenstein series modulo a prime of inertia degree 1 (but not the square of this prime). They could alternate between the two primes of inertia degree 1, but for one of these (say $\mathfrak{P}_1$) there are at least 19 forms congruent to the Eisenstein series.

For $\Oo$ the completion of the coefficient field at $\mathfrak{P}_1$ we therefore have a surplus of Eisenstein congruences, since $$1/e \sum m_{\lambda}>18>\val_{37}(\#\Oo/\eta(\mathbf{1}, 32))=2$$ (the valuation hasn't gone up in the extension from $\bfZ_p$ to $\Oo$ since the inertia degree and ramification index of the prime $\mathfrak{P}_1$ are 1). 

 It is not a priori clear that the representations associated to these cuspforms are not all isomorphic modulo $p$. But since the assumptions of Theorem \ref{main} are satisfied, we can deduce the existence of more than $\dim_{\bfF} H^1_{\Sigma}(\bfQ, \chi^{-1})$ modular lines in $H^1_{\Sigma}(\bfQ, \chi^{-1})$ and we have also proved that the Eisenstein ideal is not principal.

\bibliographystyle{amsalpha}
\bibliography{standard2}

\end{document}

%% file: makros.tex
\newcommand{\mB}{\mathcal{B}}

\newcommand{\mF}{\mathcal{F}}

\newcommand{\mJ}{\mathcal{J}}

\newcommand{\mL}{\mathcal{L}}

\newcommand{\mS}{\mathcal{S}}
\newcommand{\mT}{\mathcal{T}}

\newcommand{\fm}{\mathfrak{m}}
\newcommand{\fM}{\mathfrak{M}}

\newcommand{\fN}{\mathfrak{N}}

\newcommand{\fp}{\mathfrak{p}}
\newcommand{\fP}{\mathfrak{P}}
\newcommand{\fq}{\mathfrak{q}}

\newcommand{\fT}{\mathfrak{T}}

\newcommand{\bfC}{\mathbf{C}}

\newcommand{\bfF}{\mathbf{F}}

\newcommand{\bfP}{\mathbf{P}}
\newcommand{\bfQ}{\mathbf{Q}}

\newcommand{\bfT}{\mathbf{T}}

\newcommand{\bfZ}{\mathbf{Z}}

\newcommand{\Oo}{\mathcal{O}}

\newcommand{\OF}{\mathcal{O}_F}

\newcommand{\ov}{\overline}

\newcommand{\be}{\begin{equation}}
\newcommand{\ee}{\end{equation}}
\newcommand{\bes}{\begin{equation*}}
\newcommand{\ees}{\end{equation*}}

\newcommand{\bs}{\begin{split}}
\newcommand{\es}{\end{split}}
\newcommand{\bss}{\begin{split*}}
\newcommand{\ess}{\end{split*}}

\newcommand{\bmat}{\left[ \begin{matrix}}
\newcommand{\emat}{\end{matrix} \right]}
\newcommand{\bsmat}{\left[ \begin{smallmatrix}}
\newcommand{\esmat}{\end{smallmatrix} \right]}

\newcommand{\bml}{\begin{multline}}
\newcommand{\eml}{\end{multline}}
\newcommand{\bmls}{\begin{multline*}}
\newcommand{\emls}{\end{multline*}}

\DeclareMathOperator{\Cl}{Cl}

\DeclareMathOperator{\End}{End}
\DeclareMathOperator{\Ext}{Ext}

\DeclareMathOperator{\Frob}{Frob}
\DeclareMathOperator{\Gal}{Gal}
\DeclareMathOperator{\GL}{GL}

\DeclareMathOperator{\Hom}{Hom}

\DeclareMathOperator{\val}{val}

\newcommand{\tr}{\textup{tr}\hspace{2pt}}

%% file: settings.tex
\usepackage[all]{xy}
\usepackage{amsthm}
\usepackage{amssymb, amsfonts}
\SelectTips{cm}{10}\UseTips
\theoremstyle{plain}
\newtheorem{thm}{Theorem}
\newtheorem{prop}[thm]{Proposition}
\newtheorem{cor}[thm]{Corollary}
\newtheorem{lemma}[thm]{Lemma}

\theoremstyle{definition}
\newtheorem{definition}[thm]{Definition}

\newtheorem{rem}[thm]{Remark}

\numberwithin{thm}{section}
\numberwithin{equation}{section}

%% file: Berger_Klosin_Galois_extensions_2018.bbl
\providecommand{\bysame}{\leavevmode\hbox to3em{\hrulefill}\thinspace}
\providecommand{\MR}{\relax\ifhmode\unskip\space\fi MR }
\providecommand{\MRhref}[2]{%
  \href{http://www.ams.org/mathscinet-getitem?mr=#1}{#2}
}
\providecommand{\href}[2]{#2}
\begin{thebibliography}{WWE18}

\bibitem[BC09]{BellaicheChenevierbook}
J.~Bella{\"{\i}}che and G.~Chenevier, \emph{$p$-adic families of {G}alois
  representations and higher rank {S}elmer groups}, Ast\'erisque (2009),
  no.~324.

\bibitem[BCP97]{MAGMA97}
W.~Bosma, J.~Cannon, and C.~Playoust, \emph{The {M}agma algebra system. {I}.
  {T}he user language}, J. Symbolic Comput. \textbf{24} (1997), no.~3-4,
  235--265, Computational algebra and number theory (London, 1993).

\bibitem[BK90]{BlochKato90}
S.~Bloch and K.~Kato, \emph{{$L$}-functions and {T}amagawa numbers of motives},
  The Grothendieck Festschrift, Vol.\ I, Progr. Math., vol.~86, Birkh\"auser
  Boston, Boston, MA, 1990, pp.~333--400.

\bibitem[BK09]{BergerKlosin09}
T.~Berger and K.~Klosin, \emph{{A} deformation problem for {G}alois
  representations over imaginary quadratic fields}, Journal de l'{I}nstitut de
  {M}ath. de {J}ussieu \textbf{8} (2009), no.~4, 669--692.

\bibitem[BK13]{BergerKlosin13}
\bysame, \emph{On deformation rings of residually reducible {G}alois
  representations and ${R}={T}$ theorems}, Math. Ann. \textbf{355} (2013),
  no.~2, 481--518.

\bibitem[BK15]{BergerKlosin15}
\bysame, \emph{On lifting and modularity of reducible residual {G}alois
  representations over imaginary quadratic fields}, Int. Math. Res. Not. IMRN
  (2015), no.~20, 10525--10562.

\bibitem[BKK14]{BergerKlosinKramer14}
T.~Berger, K.~Klosin, and K.~Kramer, \emph{On higher congruences between
  automorphic forms}, Math. Res. Lett. \textbf{21} (2014), no.~1, 71--82.

\bibitem[BM16]{BillereyMenares16}
N.~Billerey and R.~Menares, \emph{On the modularity of reducible {${\rm mod}\,
  l$} {G}alois representations}, Math. Res. Lett. \textbf{23} (2016), no.~1,
  15--41.

\bibitem[BM18]{BillereyMenares18}
\bysame, \emph{Strong modularity of reducible {G}alois representations}, Trans.
  Amer. Math. Soc. \textbf{370} (2018), no.~2, 967--986.

\bibitem[Bre01]{Breuil01}
C.~Breuil, \emph{$p$-adic {H}odge theory, deformations and local {L}anglands},
  2001, http://www.ihes.fr/~breuil/PUBLICATIONS/Barcelone.pdf.

\bibitem[CHT08]{ClozelHarrisTaylor08}
L.~Clozel, M.~Harris, and R.~Taylor, \emph{Automorphy for some {$l$}-adic lifts
  of automorphic mod {$l$} {G}alois representations}, Publ. Math. Inst. Hautes
  \'Etudes Sci. (2008), no.~108, 1--181, With Appendix A, summarizing
  unpublished work of Russ Mann, and Appendix B by Marie-France Vign{\'e}ras.

\bibitem[DDT97]{DDT}
H.~Darmon, F.~Diamond, and R.~Taylor, \emph{Fermat's last theorem}, Elliptic
  curves, modular forms \& Fermat's last theorem (Hong Kong, 1993), Internat.
  Press, Cambridge, MA, 1997, pp.~2--140.

\bibitem[DF14]{DummiganFretwell14}
N.~Dummigan and D.~Fretwell, \emph{Ramanujan-style congruences of local
  origin}, J. Number Theory \textbf{143} (2014), 248--261.

\bibitem[Eis95]{Eisenbud}
D.~Eisenbud, \emph{Commutative algebra with a view toward algebraic geometry},
  Graduate Texts in Mathematics, vol. 150, Springer-Verlag, New York, 1995.

\bibitem[HR08]{HamblenRamakrishna08}
S.~Hamblen and R.~Ramakrishna, \emph{Deformations of certain reducible {G}alois
  representations. {II}}, Amer. J. Math. \textbf{130} (2008), no.~4, 913--944.

\bibitem[Klo09]{Klosin09}
K.~Klosin, \emph{Congruences among automorphic forms on {${\rm U}(2,2)$} and
  the {B}loch-{K}ato conjecture}, Annales de l'institut Fourier \textbf{59}
  (2009), no.~1, 81--166.

\bibitem[Maz77]{Mazur78}
B.~Mazur, \emph{Modular curves and the {E}isenstein ideal}, Inst. Hautes
  \'Etudes Sci. Publ. Math. (1977), no.~47, 33--186 (1978).

\bibitem[Miy89]{Miyake89}
T.~Miyake, \emph{Modular forms}, Springer-Verlag, Berlin, 1989, Translated from
  the Japanese by Yoshitaka Maeda.

\bibitem[MW84]{MazurWiles84}
B.~Mazur and A.~Wiles, \emph{Class fields of abelian extensions of {${\bf
  Q}$}}, Invent. Math. \textbf{76} (1984), no.~2, 179--330.

\bibitem[Och00]{Ochiai00}
T.~Ochiai, \emph{Control theorem for {B}loch-{K}ato's {S}elmer groups of
  {$p$}-adic representations}, J. Number Theory \textbf{82} (2000), no.~1,
  69--90.

\bibitem[Oht14]{Ohta14}
M.~Ohta, \emph{Eisenstein ideals and the rational torsion subgroups of modular
  {J}acobian varieties {II}}, Tokyo J. Math. \textbf{37} (2014), no.~2,
  273--318.

\bibitem[Oza17]{Ozawa17}
T.~Ozawa, \emph{Constant terms of {E}isenstein series over a totally real
  field}, Int. J. Number Theory \textbf{13} (2017), no.~2, 309--324.

\bibitem[Rub00]{Rubin00}
K.~Rubin, \emph{Euler systems}, Annals of Mathematics Studies, vol. 147,
  Princeton University Press, Princeton, NJ, 2000, Hermann Weyl Lectures. The
  Institute for Advanced Study.

\bibitem[Ski06]{Skinner06}
C.~M. Skinner, \emph{Main conjectures and modular forms}, Current developments
  in mathematics, 2004, Int. Press, Somerville, MA, 2006, pp.~141--161.

\bibitem[Spe18]{Spencer18}
D.~Spencer, \emph{Congruences of local origin for higher levels}, Ph.D. thesis,
  University of Sheffield, 2018.

\bibitem[Sun10]{Sun10}
Hae-Sang Sun, \emph{Cuspidal class number of the tower of modular curves
  {$X_1(Np^n)$}}, Math. Ann. \textbf{348} (2010), no.~4, 909--927.

\bibitem[SW97]{SkinnerWiles97}
C.~M. Skinner and A.~J. Wiles, \emph{Ordinary representations and modular
  forms}, Proc. Nat. Acad. Sci. U.S.A. \textbf{94} (1997), no.~20,
  10520--10527.

\bibitem[SW99]{SkinnerWiles99}
\bysame, \emph{Residually reducible representations and modular forms}, Inst.
  Hautes \'Etudes Sci. Publ. Math. (1999), no.~89, 5--126 (2000).

\bibitem[{The}18]{sagemath}
{The Sage Developers}, \emph{{S}agemath, the {S}age {M}athematics {S}oftware
  {S}ystem ({V}ersion 8.3)}, 2018, {\tt http://www.sagemath.org}.

\bibitem[Urb01]{Urban01}
E.~Urban, \emph{Selmer groups and the {E}isenstein-{K}lingen ideal}, Duke Math.
  J. \textbf{106} (2001), no.~3, 485--525.

\bibitem[Was97]{Washingtonbook}
L.~C. Washington, \emph{Introduction to cyclotomic fields}, second ed.,
  Graduate Texts in Mathematics, vol.~83, Springer-Verlag, New York, 1997.

\bibitem[WWE18]{WakeWangErickson18preprint}
P.~Wake and C.~Wang-Erickson, \emph{The {E}istenstein ideal with squarefree
  level}, Preprint (2018), arXiv: 1804.06400.

\bibitem[Yoo16]{Yoo16}
H.~Yoo, \emph{The index of an {E}isenstein ideal and multiplicity one}, Math.
  Z. \textbf{282} (2016), no.~3-4, 1097--1116.

\end{thebibliography}
